\documentclass{article}

\usepackage{amsthm,amsfonts,amsmath}
\usepackage{amssymb,amsmath,amsfonts}
\usepackage{amsthm}
\usepackage{graphicx}
\usepackage{amsthm}
\usepackage{color}

\newtheorem{teorema}{Theorem}[section]
\newtheorem{lema}[teorema]{Lemma}
\newtheorem{definicion}[teorema]{Definition}

\theoremstyle{definition}
\newtheorem{ejemplo}[teorema]{Example}
\newtheorem{remark}[teorema]{Remark}
\newcommand{\Qed}{~\hfill$\square$}

\begin{document}

\title{Principal configurations around umbilics of spacelike surfaces in null hypersurfaces of $\mathbb{R}_1^4$}
\date{July 2nd, 2019}
\maketitle

\begin{center}
\author{Matias Navarro\footnote{matias.navarro@correo.uady.mx}, Oscar Palmas\footnote{oscar.palmas@ciencias.unam.mx}, Didier A. Solis\footnote{didier.solis@correo.uady.mx}}
\end{center}

\begin{abstract}
\noindent We study the principal configurations around an isolated $\eta$-umbilical point on a generic spacelike surface $S$ immersed in a null hypersurface $M$ of Minkowski space $\mathbb{R}_1^4$ relative to a well-defined null vector field $\eta$ orthogonal to the surface $S$. In the particular case of $M$ being a null rotation hypersurface of $\mathbb{R}_1^4$ we also recover the local Darbouxian principal configurations around that kind of $\eta$-umbilical points.
\end{abstract}



\section{Introduction}\label{sec:intro}

The principal configuration of a  surface consists of a pair of foliations formed by the two families of principal curvature lines corresponding to maximal and minimal principal curvatures, being the umbilics of the surface the singular points of the foliation. In the pioneer works \cite{MR724448, MR985996} a generic class of principal configurations around isolated umbilical points of surfaces in $\mathbb{R}^{3}$ was analyzed, named {\em Darbouxian} in \cite{MR724448} in honor to G. Darboux \cite{MR1896darboux} who found in 1896 the three topologically different types which belong to that class. Since the publication of \cite{MR724448}, several authors have produced a considerable amount of research with generalizations and extensions of the results to surfaces in other ambient spaces, such as Euclidean and semi-Euclidean spaces. See, for instance \cite{MR2747949, MR961601, MR1473078, MR2105778, MR2673934, MR2080424, MR3126943, MR1900745, MR1336208} just to mention a few. In particular, P. Bayard and F. S\'anchez-Bringas study in \cite{MR2747949} the principal configurations for spacelike surfaces in Minkowski 4-space and they find the 1-jet of the differential equation of the principal curvature lines with respect to a null normal vector field. They also observe that Darbouxian principal configurations must appear generically, but they did not show the explicit dependence of these configurations on the parameters of the immersion. Here we impose the additional condition of $S$ being immersed in a null hypersurface $M\subset\mathbb{R}_1^4$ and find the coefficients, in terms of parameters of the immersion, of the differential equation of the $\eta$-principal curvature lines, where $\eta$ is a well-defined null vector field normal to $S$ and complementary to the tangent bundle of $M$ in the tangent bundle of $\mathbb{R}_1^4$. In section \ref{sec:prelim} we give the necessary preliminary concepts on the geometry of a pair $(M,S)$ formed by a null hypersurface $M^{n+1}$ of $\mathbb{R}_1^{n+2}$ and a spacelike submanifold $S$ of dimension $n$ immersed in $M$, by means of two shape operators defined by independent null vector fields $\eta$ and $\xi$ orthogonal to $S$, following the approach in \cite{MR1313822}. If the null hypersurface $M$ is the light cone $\Lambda^{n+1}$ then any hypersurface $S$ of $M$ is totally umbilical with respect to $\xi$ and the $\xi$-principal configuration is trivial. For the $\eta$-shape operator, it was shown by the authors in \cite{MR3126943} the existence of non-totally umbilical spacelike surfaces $S\subset \Lambda^3 \subset \mathbb{R}_1^4$ and therefore it makes sense to study the $\eta$-principal curvature lines around isolated $\eta$-umbilical points of spacelike surfaces immersed in $\Lambda^3 \subset \mathbb{R}_1^4$. Also in \cite{MR3126943} the authors obtained the explicit dependence of the coefficients of the differential equation of these curvature lines on the immersion. In section \ref{sec:null} we generalize this framework for any null hypersurface $M\subset \mathbb{R}_1^4$. Finally, in section \ref{sec:nullrotation} we classify the local Darbouxian principal configurations in a neighbourhood of isolated $\eta$-umbilical points of spacelike surfaces immersed in all possible null rotation hypersurfaces of $\mathbb{R}_1^4$ (see Theorem \ref{teo:Darbouxianos}). 

\section{Preliminaries}\label{sec:prelim}

In this section we follow closely the notation and results in \cite{MR3126943}. The \emph{Minkowski $(n+2)$-space} $\mathbb{R}_{1}^{n+2}$ is the $(n+2)$-dimensional vector space $\mathbb{R}^{n+2}$ endowed with the scalar product
\begin{equation}\label{metric}
\langle p,q \rangle= -u_{0}v_{0}+\sum_{i=1}^{n+1} u_{i}v_{i},
\end{equation}
where $(u_{0},u_{1},\dots,u_{n+1})$ and $(v_{0},v_{1},\dots,v_{n+1})$ are
respectively the coordinates of $p$ and $q$ relative to the canonical basis $e_0,e_1,\dots,e_{n+1}$ of $\mathbb{R}^{n+2}$. 

Throughout this work $M$ will denote a \emph{null} (or \emph{lightlike})
hypersurface of $\mathbb{R}_{1}^{n+2}$, that is, a $(n+1)$-submanifold such that
the restriction of the scalar product (\ref{metric}) to the tangent bundle $TM $ is degenerate. 
This degeneracy condition is equivalent to the existence of a vector field $\xi\in\Gamma(TM)$ everywhere different from zero such that $\langle\xi,X\rangle =0$ for each $X\in\Gamma(TM)$.

The $(n+1)$-dimensional \emph{light cone} of $\mathbb{R}_{1}^{n+2}$ is the null hypersurface defined by
\begin{equation*}
\Lambda^{n+1}=\{\ p\in\mathbb{R}_{1}^{n+2}\ \vert\ \langle p, p\rangle=0, p\ne0\ \}.
\end{equation*}

Given a null hypersurface $M\subset\mathbb{R}_1^{n+2}$, we will consider a
\emph{spacelike} hypersurface $S$ of $M$, that is, $\dim S=n$ and the
restriction of the scalar product (\ref{metric}) to the tangent bundle $TS$ is positive definite.

In order to define the basic geometrical objects related to the pair $(M,S)$, we will split the tangent bundle $T\mathbb{R}_1^{n+2}$ into three vector bundles. From \cite{MR1313822}, we know that for each point $p\in S$ there exists a neighbourhood $\mathcal{U}$ in $S$ and a unique vector field 
$\eta \in \Gamma(T\mathbb{R}_1^{n+2}\vert _{\mathcal{U}})$ such that
\begin{equation} \label{eta}
\langle\xi, \eta\rangle=1, \quad \langle\eta,\eta\rangle=\langle\eta,X\rangle=0,
\end{equation}
for each $X\in\Gamma(TS\vert _{\mathcal{U}})$. Using this null vector field $\eta$
we can write, for each $p\in \mathcal{U}$, the \emph{transversal} decomposition
\begin{equation}  \label{eq:descomposicion0}
T_{p}\mathbb{R}_{1}^{n+2}=T_{p}M\oplus \mathrm{span}(\eta_p).
\end{equation}
Additionally, we decompose $TM$ as
\begin{equation}  \label{eq:descomposicion1}
T_{p}M=T_{p}S\oplus_{\mathrm{orth}} \mathrm{span}(\xi_p),
\end{equation}
so that
\begin{equation*}
T_{p}\mathbb{R}_{1}^{n+2}=T_{p}S\oplus_{\mathrm{orth}}(\mathrm{span}(\xi_{p})\oplus\mathrm{span}(\eta_{p})).
\end{equation*}

Following \cite{MR2598375}, the Gauss-Weingarten formulae corresponding to these decompositions are given using the connections defined as follows. Denote by $\widetilde\nabla$
the semi-Riemannian connection in $\mathbb{R}_{1}^{n+2}$, and let $X,Y\in\Gamma(TM)$. Using (\ref{eq:descomposicion0}) we
write the \emph{first Gauss formula} as
\begin{equation}  \label{eq:gauss1}
\widetilde\nabla_{X}Y=\nabla_{X}Y+h(X,Y),
\end{equation}
where $\nabla$ denotes the induced connection in $M$ and $h$ is called the
\emph{second fundamental form} of $M$ in $\mathbb{R}_1^{n+2}$.

On the other hand, if $X\in\Gamma(TM)$, we use again (\ref{eq:descomposicion0}) to write the \emph{first Weingarten formula}
\begin{equation}  \label{eq:weingarten1}
\widetilde\nabla_{X} \eta=-A_{\eta}X+\nabla_{X}^t \eta,
\end{equation}
where $A_\eta$ is the \emph{shape operator} with respect to $\eta$ and $\nabla^t$ is the induced
transversal connection of $M$ in $\mathbb{R}_{1}^{n+2}$.

Let $P:TM\to TS$ be the orthogonal projection relative to the decomposition (\ref{eq:descomposicion1}). The \emph{second Gauss-Weingarten formulae} are
\begin{equation}  \label{eq:gauss2}
\nabla_{X}PY=\nabla_{X}^*PY+h^*(X,PY)
\end{equation}
and
\begin{equation}  \label{eq:weingarten2}
\nabla_{X} \xi=-A_{\xi}^*X+\nabla_{X}^{*t} \xi,
\end{equation}
for $X,Y\in\Gamma(TM)$, where $\nabla^*$ and $\nabla^{*t}$ are linear
connections which will not be used here; we are interested only on the
operator $A_\xi^*$ and the form $h^*$, which are called the \emph{screen
shape operator} and the \emph{screen second fundamental form}, respectively.
It is easy to check that
\begin{equation}  \label{eq:2ff}
\langle h^*(X,PY),\eta \rangle = \langle A_\eta X, PY \rangle;
\end{equation}
compare for example, equations (2.1.21) and (2.1.26) in \cite{MR2598375}.

\begin{definicion}
\label{def:ptoumbilico} Given a normal vector field $\nu$ to the surface $S$, we say that a point $p\in S$ is \textbf{$\nu$-umbilical} if there exists a real-valued function 
$k$ on $S$ such that $A_{\nu}(X(p))=k(p)X(p)$ for each $X\in \Gamma(TS)$. 
If every point of $S$ is $\nu$-umbilical we say that $S$ is \textbf{totally umbilical} with respect to the normal vector field $\nu$.
\end{definicion}

\begin{ejemplo}
Note that the position vector field $\xi\in\mathbb{R}_{1}^{n+2}$ satisfies $\widetilde \nabla_{X}\xi=X$ for any $X\in\Gamma(T\mathbb{R}_{1}^{n+2})$. In
particular, for \emph{any} hypersurface $S\subset\Lambda^{n+1}$ and any 
$X\in\Gamma(T\Lambda^{n+1})$ we have
\begin{equation*}  \label{eq:xiumbilica}
A_{\xi}^*X=-PX.
\end{equation*}
The last equation implies that for $X\in \Gamma(TS)$ the screen shape operator $A_\xi^*$ is a multiple of the identity and $S$ is totally umbilical with respect to $\xi$. Therefore, there are no isolated $\xi$-umbilical points on any surface $S\subset \Lambda^3 \in \mathbb{R}_1^4$.
\end{ejemplo}

\section{Spacelike surfaces in null hypersurfaces of $\mathbb{R}_1^4$}

\label{sec:null}

As was observed in \cite{MR2673934}, the concept of principal curvature lines is derived from the existence of a self-adjoint operator with respect to a given metric and with real eigenvalues. For $\eta$ defined by (\ref{eta}) in section \ref{sec:prelim}, $A_\eta$ is $\Gamma(TS)$-valued and the distribution $TS$ is integrable; therefore, by Theorem 2.2.6 in \cite{MR2598375} it follows that the shape operator $A_{\eta }$ restricted to $TS$ is self-adjoint. Then, for each $p\in S$ there is an orthonormal basis in $T_{p}S$ of eigenvectors of $A_{\eta }$ with corresponding real eigenvalues, since the metric (\ref{metric}) on a spacelike surface is positive definite. These eigenvalues are called $\eta $-\emph{principal curvatures} at $p$ and, according to the definition of umbilicity given in section \ref{sec:prelim}, a point in $S$ is $\eta $-\emph{umbilical} if both $\eta $-principal curvatures coincide at that point. On the other hand,
for any non-umbilical point there are two $\eta $-\emph{principal directions} $E_1,E_2$
that define two smooth line fields by the equation $A_{\eta }E_i=k_i E_i$, whose
integral lines $L_i$ are called the $\eta $-\emph{principal curvature lines} for $i=1,2$. An isolated $\eta $-umbilical point $p\in S$ together with the two families $L_1, L_2$ of $\eta$-principal curvature lines on the surface $S$ in a neighborhood of $p$ form the {\em local $\eta$-principal configuration} at $p\in S$. Let $\mathcal{U}_{\eta}(S)$ be the set of $\eta$-umbilical points of $S$ and $L_i(S)$ the family of $\eta$-principal curvature lines which correspond to the $\eta$-principal direction $E_i$. The  triple $(\mathcal{U}_{\eta}(S),L_1(S),L_2(S))$ is the {\em $\eta$-principal configuration}. 

As in \cite{MR724448} and \cite{MR2007065}, we say that two of our spacelike surfaces $S_1$ and $S_2$ with $\eta$-principal configurations
\[
(\mathcal{U}_{\eta}(S_1),L_1(S_1),L_2(S_1)) \quad \text{and} \quad (\mathcal{U}_{\eta}(S_2),L_1(S_2),L_2(S_2))
\]
are {\em $\eta$-principally equivalent} if there is a homeomorphism $h:S_1\rightarrow S_2$ such that $h(\mathcal{U}_{\eta}(S_1))=\mathcal{U}_{\eta}(S_2)$ and $h$ sends $\eta$-principal curvature lines $L_i(S_1)$ into $\eta$-principal curvature lines $L_i(S_2)$ respectively, for each $i=1,2$. In the seminal work  \cite{MR0730276} it was proved that every immersion of a compact oriented smooth 2-manifold into $\mathbb{R}^3$ can be arbitrarily $C^2$-approximated by smooth immersions whose principal configurations are stable under $C^3$-sufficiently small perturbations of the immersion in the sense introduced in \cite{MR724448}. On the other hand, the local classification of generic multi-valued direction fields in the plane up to homeomorphism was made by Davydov in \cite{MR0800916} following the approach of Arnold \cite{MR0947141}. In the same line of research, Bruce and Fidal \cite{MR985996} consider those bivalued direction fields whose direction pairs are mutually orthogonal at each point and they give local classification up to homeomorphism of solution curves of binary differential equations of the form
\begin{equation}\label{BDE}
A(x,y) \ dy^2 + 2B(x,y) \ dy \ dx - A(x,y) \ dx^2=0,
\end{equation}
where $A$ and $B$ are smooth functions which vanish at the origin. Besides, they show that the homeomorphism is a diffeomorphism away from certain directions though the origin. As we are going to show in this section, the $\eta$-principal curvature lines of our spacelike surface $S$ satisfy binary differential equations of that kind. The main result of \cite{MR985996} is that, under certain conditions on the 1-jet of the functions $A, B$, there is a germ of a homeomorphism of the $xy$-plane onto itself taking the integral curves of (\ref{BDE}) to the integral curves of one of three normal forms which are principally equivalent to the Darbouxian types described in \cite{MR724448}. There is a considerable overlap of results in both papers but with different blowing-up constructions. See also \cite{MR1328597, MR1449143, MR961601, MR1473078} where there are similar classification results for equations like (\ref{BDE}).

The differential equation of our $\eta$-principal curvature line $c(t)$ is given by
\begin{equation}\label{rodrigues}
A_{\eta }(c^{\prime }(t))=k(t)c^{\prime}(t),
\end{equation}
which may be expressed in local coordinates as follows. Let $S$ be a spacelike surface immersed in a null hypersurface $M$ of Minkowski space $\mathbb{R}_1^4$ and let $\varphi$ be a parametrization of an open neighborhood $\mathcal{U}\subset S$ with local coordinates $(x,y)$. For each $p=\varphi(x,y)$, the associated basis of $T_{p}S$ is given by $\varphi_{x}=\partial\varphi /\partial x$ and $\varphi_{y}=\partial \varphi /\partial y$. In view of equations (\ref{eq:gauss2}) and (\ref{eq:2ff}), the \emph{coefficients of the screen second fundamental form} satisfy
\begin{equation}\label{e}
e_\eta =\langle h^*(\varphi_x, \varphi_x), \eta\rangle= \langle A_\eta (\varphi_x), \varphi_x \rangle= \langle \nabla_{\varphi_x}\varphi_x, \eta \rangle = \langle \varphi_{xx}, \eta \rangle,
\end{equation}
\begin{equation}\label{f}
f_\eta =\langle h^*(\varphi_x, \varphi_y), \eta\rangle= \langle A_\eta (\varphi_x), \varphi_y \rangle=\langle \nabla_{\varphi_x}\varphi_y, \eta \rangle = \langle \varphi_{xy}, \eta \rangle,
\end{equation}
\begin{equation}\label{g}
g_\eta =\langle h^*(\varphi_y, \varphi_y), \eta\rangle= \langle A_\eta (\varphi_y), \varphi_y \rangle=\langle \nabla_{\varphi_y}\varphi_y, \eta \rangle = \langle \varphi_{yy}, \eta \rangle . 
\end{equation}

\noindent On the other hand, because $A_\eta(X)\in T_pS$ for each $X\in T_pS$, there exist functions $a_{ij}$ such that
\begin{equation*}
A_\eta(\varphi_x)=a_{11}\varphi_x+a_{21}\varphi_y, \ \ A_\eta(\varphi_y)=a_{12}\varphi_x+a_{22}\varphi_y,
\end{equation*}
which imply that the coefficients (\ref{e}), (\ref{f}), (\ref{g}) satisfy
\begin{eqnarray*}
e_\eta &=&\langle A_\eta(\varphi_x),\varphi_x \rangle = a_{11}E+a_{21}F, \\
f_\eta &=&\langle A_\eta(\varphi_x),\varphi_y \rangle = a_{11}F+a_{21}G, \\
f_\eta &=&\langle A_\eta(\varphi_y),\varphi_x \rangle = a_{12}E+a_{22}F, \\
g_\eta &=&\langle A_\eta(\varphi_y),\varphi_y \rangle = a_{12}F+a_{22}G,
\end{eqnarray*}
where
\begin{equation}\label{EFG}
E=\langle \varphi_x, \varphi_x \rangle, \ \ F=\langle \varphi_x, \varphi_y \rangle, \ \ G=\langle \varphi_y, \varphi_y \rangle,
\end{equation}
are the coefficients of the {\it first fundamental form} of the immersion. The functions $a_{ij}$ can be obtained in terms of $E,F,G,e_\eta,f_\eta,g_\eta$ from the above system. 
Then, if $c(t)=\varphi(x(t),y(t))$ satisfies (\ref{rodrigues}) its coordinate functions are solutions of the system 
\begin{eqnarray*}
a_{11} x' + a_{12} y'&=&k x', \\
a_{21} x' + a_{22} y'&=&k y'.
\end{eqnarray*}
Eliminating the parameter $k$ from the above system, the equation (\ref{rodrigues}) is written in local coordinates $(x,y)$ of the open subset $\varphi^{-1}({\mathcal{U}})\subset \mathbb{R}^2$ as
\begin{equation}\label{ABC}
a_{12}(y')^2+(a_{11}-a_{22})x'y'-a_{21}(x')^2=0,
\end{equation}
where
\begin{eqnarray*}
a_{12}&=&f_{\eta }G-g_{\eta }F, \\
a_{11}-a_{22}&=&e_{\eta }G-g_{\eta }E, \\
a_{21}&=&f_{\eta }E-e_{\eta }F.
\end{eqnarray*}
In a more usual notation of several authors (see \cite{MR724448}, \cite{MR985996}, \cite{MR1473078}) equations like (\ref{ABC}) are written in the form
\[
A(x,y) \ dy^2 + B(x,y) \ dx \ dy + C(x,y) \ dx^2 = 0,
\]
which we adopt from now on.

Now, in order to build a parametrization of a null hypersurface $M$ of $\mathbb{R}_1^4$, we consider a smooth function $f:U\subset \mathbb{R}^{2}\rightarrow \mathbb{R}$ where $U$ is a neighborhood of the origin and let
\begin{equation*}
N=\frac{1}{\sqrt{1+f_{x}^{2}+f_{y}^{2}}}\left( -f_{x},-f_{y},1\right)
\end{equation*}
be a unit normal vector to the graph of $f$ in $\mathbb{R}^{3}$; we use this to define a null vector of 
$\mathbb{R}_1^4$ by
\begin{equation*}
\xi (x,y)=(1,N(x,y)).
\end{equation*}
Then a local parametrization of a null hypersurface $M\subset \mathbb{R}_1^4$ is given by
\begin{equation}\label{parametrizacionM}
\varphi (x,y,t)=(1,x,y,f(x,y))+t\,\xi (x,y).
\end{equation}
Taking $t$ as a smooth germ of a non constant function $g:U\subset \mathbb{R}^2\rightarrow \mathbb{R}$ we obtain a generic surface $S$ contained in $M$ and parametrized by
\begin{equation}\label{parametrizacionS}
\Phi (x,y)=(1,x,y,f(x,y))+g(x,y)\,\xi (x,y),
\end{equation}
with $g(0,0)=g_x(0,0)=g_y(0,0)=0$. Now, let $p=\Phi(0,0)$ be an isolated $\eta$-umbilical point of $S$ for the vector field $\eta$ defined by (\ref{eta}) in section \ref{sec:prelim}. 
We note that ${\Phi_x,\Phi_y}$ form an orthonormal basis of $T_pS$ and without loss of generality we can assume that the tangent plane to the graph of $f$ at $(0,0,f(0,0))$ in $\mathbb{R}^3$ coincides with the $(x,y)$-plane, i.e.
\begin{equation*}
f_x(0,0)=f_y(0,0)=0.
\end{equation*}
Then, at the origin we have
\begin{center}
$\det
\left( 
\begin{array}{crl}
E &F \\
F &G \\
\end{array}
\right)
=1,
$
\end{center}
which implies that $\Phi$ defines an immersion of a neighborhood $U\subset \mathbb{R}^2$ of the origin  into an spacelike surface $S=\Phi(U)\subset M\subset \mathbb{R}_1^4$. 

Now we notice that the 1-jet of the vector field $\eta$ together with the 3-jet of the local parametrization $\Phi$ determine generically the differential equation of $\eta$-principal curvature lines. The proof goes in exactly the same way as the proof of Lemma 2.4 in \cite{MR1336208}. However, the calculations needed to find the 1-jet of our normal vector field $\eta$ and the $3$-jet of our parametrization (\ref{parametrizacionS}) are much more elaborate, as we are going to show in the following.

\begin{lema}\label{lema:1jetnormal}
The 1-jet of the vector field $\eta$ is given by
\begin{equation}\label{1jetnormal}
(-1/2,-kx,-ky,1/2),
\end{equation}
where $k$ is the $\eta$-principal curvature at $p=\Phi(0,0)$.
\end{lema}

\begin{proof} By imposing the conditions
\begin{equation*}
\langle \eta_0, \xi \rangle=1, \ \ \langle \eta_0, \eta_0 \rangle=0, \ \ \langle \eta_0, \Phi_x\rangle = \langle \eta_0, \Phi_y\rangle =0
\end{equation*}
at $(0,0)$ we obtain the null normal vector $\eta _0=(-1/2,0,0,1/2)$. Now, the 3-jet of the function $g(x,y)$ can be written, after an adequate rotation of the $(x,y)$-plane, as
\begin{equation*}
j^3g(x,y)=\frac{1}{2}(g_{xx}(0,0)x^2 + g_{yy}(0,0)y^2)+\frac{\delta}{6}x^3+\frac{\epsilon}{2}x^2y+\frac{\zeta}{2}xy^2+\frac{\lambda}{6}y^3.
\end{equation*}
The $\eta$-umbilicity at $p=\Phi(0,0)$ implies that the $\eta$-principal curvature $k$ at $p$ coincides with the coefficients $e_\eta(0,0)$ and $g_\eta(0,0)$ of the screen second fundamental form relative to $\eta$ at the origin. Besides, we have $f_\eta(0,0)=0$. On the other hand,
\begin{eqnarray*}
e_\eta(0,0)&=&\langle \Phi_{xx}(0,0),\eta_0\rangle=\frac{1}{2}f_{xx}(0,0)+g_{xx}(0,0), \\
f_\eta(0,0)&=&\langle \Phi_{xy}(0,0),\eta_0\rangle=\frac{1}{2}f_{xy}(0,0), \\
g_\eta(0,0)&=&\langle \Phi_{yy}(0,0),\eta_0\rangle=\frac{1}{2}f_{yy}(0,0)+g_{yy}(0,0).
\end{eqnarray*}
Consequently, we write the second order derivatives of $f(x,y)$ at the origin as
\begin{eqnarray*}
f_{xx}(0,0) &=&2(k-g_{xx}(0,0)), \\
f_{xy}(0,0) &=&0, \\
f_{yy}(0,0)&=&2(k-g_{yy}(0,0)).
\end{eqnarray*}
The 3-jet of $f(x,y)$ at the origin becomes
\begin{equation*}
j^3f(x,y)=f(0,0)+k(x^2+y^2)-g_{xx}(0,0)x^2-g_{yy}(0,0)y^2+\frac{a}{6}x^3+\frac{d}{2}x^2y+\frac{b}{2}xy^2+\frac{c}{6}y^3.
\end{equation*}
With these 3-jets of $f$ and $g$ we obtain the following 3-jets of the components of our parametrization (\ref{parametrizacionS}):
\begin{eqnarray*}
j^3\Phi_1(x,y)&=&1+g_{xx}(0,0)x^2/2+g_{yy}(0,0)y^2/2+\delta x^3/6+\epsilon x^2y/2\\[0.2cm]
                         & &\ \ \ \ \ \ \ \ \ \ \ \ \ \ \ \ \ \ \ \ \ \ \ \ \ \ \ \ \ \ \ \ \ \ \ \ \ \ \ \ \ \ \ \ \ \ \ \ \ \ \ \ \ +\zeta xy^2/2+\lambda y^3/6, \\[0.2cm]
j^3\Phi_2(x,y)&=&x+g_{xx}(0,0)(g_{xx}(0,0)-k)x^3+g_{yy}(0,0)(g_{xx}(0,0)-k)xy^2,\\[0.4cm]
j^3\Phi_3(x,y)&=&y+g_{xx}(0,0)(g_{yy}(0,0)-k)x^2y+g_{yy}(0,0)(g_{yy}(0,0)-k)y^3,\\[0.4cm]
j^3\Phi_4(x,y)&=&(k-g_{xx}(0,0)/2)x^2+(k-g_{yy}(0,0)/2)y^2+(a+\delta)x^3/6\\[0.2cm]
                  &  &\ \ \ \ \ \ \ \ \ \ \ \ \ \ \ \ \ \ +(d+\epsilon)x^2y/2+(b+\zeta)xy^2/2+(c+\lambda)y^3/6.
\end{eqnarray*}
Now, in order to find the 1-jet of a null vector field which coincides with $\eta_0$ at $(0,0)$ and is orthogonal to the 3-jet of (\ref{parametrizacionS}) in a neighbourhood of $p$, we introduce
\begin{equation*}
j^1\eta(x,y)=(-1/2+m_1 x + n_1 y, \ \ m_2 x + n_2 y, \ \ m_3 x + n_3 y, \ \ 1/2+m_4 x + n_4 y) 
\end{equation*}
into the conditions
\begin{equation*}
\langle j^1\eta, j^1\xi \rangle=1, \quad \langle j^1\eta, j^1\eta \rangle=\left \langle j^1\eta, \dfrac{\partial(j^3\Phi)}{\partial x}\right \rangle = \left \langle j^1\eta, \dfrac{\partial(j^3\Phi)}{\partial y}\right \rangle =0
\end{equation*}
obtaining a system of equations for the parameters $m_i,n_i$, $i=1,2,3,4$, whose solution, after long but straightforward computations, gives us the 1-jet (\ref{1jetnormal}) for $\eta(x,y)$. \end{proof}

\begin{teorema}\label{teo:ecuadif}
The differential equation of $\eta$-principal curvature lines, using the previous notations and assumptions, has 1-jet given by
\begin{equation}\label{unojetABC}
(a_{1}x+a_{2}y)\,dy^{2}+(b_{1}x+b_{2}y)\,dx\,dy-(a_{1}x+a_{2}y)\,dx^{2}=0,
\end{equation}
where
\begin{equation}
\begin{array}
[c]{l}
a_1=d+2\epsilon,\\ [0.2cm]
a_2=b+2\zeta,\\ [0.2cm]
b_1=a-b+2(\delta-\zeta), \\ [0.2cm]
b_2=d-c+2(\epsilon-\lambda).\label{dependence}
\end{array}
\end{equation}
\end{teorema}

\begin{proof}
From the 1-jet (\ref{1jetnormal}) of $\eta(x,y)$ and the 3-jet $j^3\Phi$ of the immersion 
$\Phi$ obtained in lemma \ref{lema:1jetnormal}, we find the coefficients $e_\eta(x,y)$, $f_\eta(x,y)$, $ g_\eta(x,y)$ of the screen second fundamental form  in the neighborhood $U$. Using this and the coefficients $E(x,y)$, $F(x,y)$, $G(x,y)$ of the first fundamental form of the immersion, equation (\ref{unojetABC}) follows from (\ref{ABC}). \end{proof}

Now, in order to define the Darbouxian principal configurations  for a generic spacelike surface $S$ we follow the standard approach given in \cite{MR724448} and \cite{MR1336208} for surfaces in $\mathbb{R}^3$ and $\mathbb{R}^4$, respectively. See also \cite{MR2007065} and \cite{MR2532372} for more accesible surveys, going from the classical results in $\mathbb{R}^3$ to more recent developments.

For an immersion $\Phi:U\subset \mathbb{R}^2\rightarrow M\subset \mathbb{R}_1^{4}$ of the spacelike surface $S=\Phi(U)$ we consider the projective line bundle $PS$ over $S$ with projection $\pi$.
The two coordinate charts
\begin{equation*}
(x,y;w=dx/dy)  \text{ and }  (x,y;z=dy/dx) 
\end{equation*}
cover $\pi^{-1}(S)$. The equation (\ref{ABC}) defines in $PS$ a surface
\begin{equation}\label{L}
L(\Phi,\eta)=\mathcal{F}^{-1}(0), 
\end{equation}
where $\mathcal{F}:PS\rightarrow\mathbb{R}$ can be written in local coordinates $(x,y,z)$ as
\begin{equation}\label{F}
\mathcal{F}(x,y,z)=A(x,y)\,z^{2}+B(x,y)\,z+C(x,y). 
\end{equation}
Let us denote by $\mathcal{U}_\eta$ the set of $\eta$-umbilical points of
$S$. Away from $\pi^{-1}(\mathcal{U}_\eta)$ the set (\ref{L}) is a regular surface of $PS$. In fact, it is a double covering of $S \setminus \mathcal{U}_\eta$ and our assumption that $p=\Phi(0,0)$ is an isolated $\eta$-umbilical point implies that the projective line $\pi^{-1}(p)$ is contained in $L(\Phi,\eta)$. The common locus of the curves $A(x,y)=0$, $B(x,y)=0$ represents the set $\mathcal{U}_\eta$ in the local coordinates $(x,y)$ of the surface $S$ and because we are interested in isolated umbilical points, we establish the following condition.

\begin{definicion}[Condition ({\bf T})] The pair $(\Phi,\eta)$ satisfies the \textbf{transversality condition} at $p\in \mathcal{U}_\eta$ if the curves $A(x,y)=0$ and $B(x,y)=0$  intersect transversally at 
$\Phi^{-1}(p)$.
\end{definicion}

It is easy to verify that condition ({\bf T}) is equivalent to the regularity of $L(\Phi,\eta)$ along 
$\pi^{-1}(p)$; see remark 1 in section 2 of \cite{MR2105778}.

Denoting by $\mathcal{F}_x$, $\mathcal{F}_y$ and $\mathcal{F}_z$ the partial derivatives of $\mathcal{F}$ with respect to $x,y$ and $z$, the {\em Lie-Cartan} vector field of the pair $(\Phi,\eta)$ is defined in local coordinates $(x,y,z)$ by
\begin{equation} \label{XLieCartan}
\mathcal{X}=\mathcal{F}_{z}\frac{\partial}{\partial x}+z\mathcal{F}_{z}
\frac{\partial}{\partial y} - (\mathcal{F}_{x} + z\mathcal{F}_{y})\frac{\partial}{\partial z}
\end{equation}
and has the following properties:

\begin{description}
\item(i) $\mathcal{X}$ is tangent to $L(\Phi,\eta)$,

\item(ii) $\pi_\ast(\mathcal{X})$ vanishes only at the origin,

\item(iii) if $(x,y,z)\in L(\Phi,\eta)$ with $(x,y)\neq(0,0)$, then
$\pi_{\ast}(\mathcal{X}(x,y,z))$ generates the $\eta$-principal curvature line
with direction $(1,z)$.
\end{description}
These properties of the vector field $\mathcal{X}$ can be used to obtain a blow up of the 
$\eta$-umbilical point with $\mathcal{X}$ tangent to the pull back of the $\eta$-principal curvature lines, which are the solution curves of $\mathcal{X}$ on the surface $L(\Phi,\eta)$ near the singular points of $\mathcal{X}$. 

Observe that, since $\mathcal{F}_z(0,0,z)=0$ for any $z$, the first two components of $\mathcal{X}$ vanish on the fiber $\pi^{-1}(p)$. Therefore, the singularities of the Lie-Cartan vector field $\mathcal{X}$ over $\pi^{-1}(p)$ are the points $(0,0,z)$ where $z$ is a root of the third component of $\mathcal{X}$, namely
\begin{equation}\label{fz}
f(z)=(\mathcal{F}_{x}+ z\mathcal{F}_{y})(0,0,z).
\end{equation}

Using the 1-jets of $A(x,y)$ and $B(x,y)$, the 1-jet of $\mathcal{F}$ is given by
\begin{align*}
\mathcal{F}_1(x,y,z)  &  =((d+2\epsilon)x+(b+2\zeta)y)z^2\nonumber\\
&  +((a-b+2(\delta-\zeta))x+(d-c+2(\epsilon-\lambda))y)z\nonumber\\
&  -(d+2\epsilon)x-(b+2\zeta)y.
\end{align*}
The function (\ref{fz}) with this 1-jet of $\mathcal{F}$ is the cubic polynomial
\begin{equation}\label{cubic}
f_1(z)=(b+2\zeta)z^3+(2d-c+2(2\epsilon-\lambda))z^2+(a-2b+2(\delta-2\zeta))z-d-2\epsilon.
\end{equation}

As was noted in \cite{MR2007065} and \cite{MR1336208} for surfaces in $\mathbb{R}^3$ and 
$\mathbb{R}^4$ respectively, we also have that the roots $z_i$ of the cubic polynomial (\ref{cubic}) determine the possible directions along which the $\eta$-principal curvature lines can approach the umbilic point $p$ in the following sense: the saddle separatrices of the Lie-Cartan vector field $\mathcal{X}$ which are orthogonal to the fiber $\pi^{-1}(p)$ at its singularities $(0,0,z_i)$ project into the surface $S$ to the umbilical separatrices, which are the $\eta$-principal curvature lines approaching the umbilic point $p$ in a definite direction in $T_pS$ defined by $(1,z_i)$. This is due to property (iii) of the Lie-Cartan vector field $\mathcal{X}$ listed above. See figure 1. For another approach to similar problems which have equivalent results see \cite{MR985996}, \cite{MR1328597} and \cite{MR1449143}.

\begin{figure}[h]
\label{fig} \centering
\includegraphics[scale=0.3]{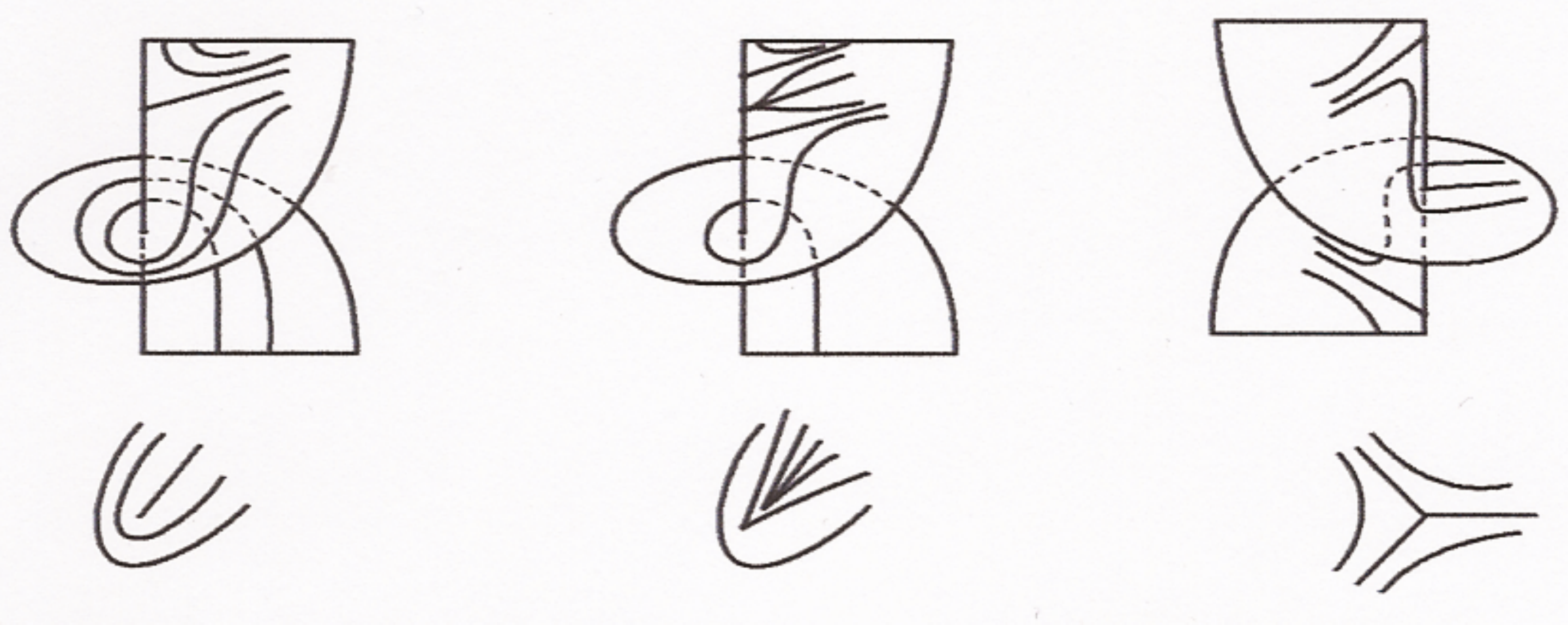}
\caption{Lifted umbilics to singularities of Lie-Cartan vector fields.}
\end{figure}

Notice that the constant term in (\ref{cubic}) may be set equal to zero by a change of coordinates as given, for example, in \cite{MR961601}, however this procedure seems to be impracticable in our context, because of the enormous amount of calculations required. Nevertheless, for null rotation hypersurfaces of $\mathbb{R}_1^4$ as ambient spaces for our surface $S$, we can always obtain explicitly the cubic polynomial (\ref{cubic}); see section \ref{sec:nullrotation}.  Returning to our calculations, we observe that for each root $z_i$ of the cubic polynomial (\ref{cubic}) the eigenvalues of the linear part of $\mathcal{X}$ at the singular point $(0,0,z_i)$ are
\begin{equation}
\begin{array}
[c]{l}
\beta_1(z_i)=0,\\[0.2cm]
\beta_2(z_i)=a-b+2(\delta-\zeta)+(3d-c+2(3\epsilon-\lambda))z_i+2(b+2\zeta)z_i^2,\\[0.2cm]
\beta_3(z_i)=2b-a+2(2\zeta-\delta)+(2(c-2d)+4(\lambda-2\epsilon))z_i-3(b+2\zeta)z_i^2.\\ \label{betas}
\end{array}
\end{equation}
It is easy to see that $(1,z_i,0)$ and $(0,0,1)$ are eigenvectors associated to the non zero eigenvalues (\ref{betas}) and they form a basis of the tangent plane to the surface $L(\Phi,\eta)$ at the singular point $(0,0,z_i)$. Therefore, the type of singularity of $\mathcal{X}$ at each singular point on the projective line $\pi^{-1}(p)$ is determined by $\beta_2(z_i)$ and $\beta_3(z_i)$. 
The {\em simple} $\eta$-umbilical points are defined as follows.

\begin{definicion}\label{simple}
Let $p$ be an isolated $\eta$-umbilical point of the surface $S$ with $\eta$-principal curvature lines around $p$ defined by (\ref{ABC}). In terms of (\ref{parametrizacionS}) and (\ref{1jetnormal}) we call 
$p=\Phi(x,y)$ \textbf{simple} if it is a non degenerate minimum of the function
\begin{equation*}
B^{2}(x,y)-4A(x,y)C(x,y).
\end{equation*}
\end{definicion}

For the differential equation (\ref{unojetABC}) and its dependence on parameters (\ref{dependence}), the simple $\eta$-umbilical points of $S$ are characterized by the following two inequalities:
\begin{equation}
4(d+2\epsilon)^2+(a-b+2(\delta-\zeta))^2>0,\label{simple1}
\end{equation}
\begin{equation}
((d+2\epsilon)(d-c+2(\epsilon-\lambda))-(b+2\zeta)(a-b+2(\delta-\zeta)))^2>0,\label{simple2}
\end{equation}

\begin{remark}\label{condicionT}
In terms of (\ref{unojetABC}), the generic condition $a_1\neq 0$ guarantees that (\ref{simple1}) holds true and conditon \textbf{(T)} is equivalent to $a_1 b_2-a_2 b_1\neq 0$, which imply (\ref{simple2}). Therefore, $a_1\neq 0$ and condition \textbf{(T)} ensure that the $\eta$-umbilical point is simple.
\end{remark}

\begin{definicion}\label{Darbouxian}
Let $p$ be a simple $\eta$-umbilical point of a surface $S$ immersed in a null hypersurface $M\subset \mathbb{R}_1^{4}$ with the cubic polynomial (\ref{cubic}) associated to $(\Phi,\eta)$. Then $p$ is called \textbf{Darbouxian} if (\ref{cubic}) has only simple roots and condition {\bf (T)} holds.
\end{definicion}

Following \cite{MR724448}, the Darbouxian $\eta$-principal configurations are classified into three types depending on the hyperbolicity of the Lie-Cartan vector field (\ref{XLieCartan}) at its singularities over the projective line $\pi^{-1}(p)$ as follows.

\begin{definicion}\label{D1D2D3}
A Darbouxian $\eta$-umbilical point of $(\Phi,\eta)$ is named:
\begin{itemize}
\item[$(\mathbf{D_{1}})$] if the cubic polynomial (\ref{cubic}) has only one real root and the corresponding singularity is a saddle point of $\mathcal{X}$,

\item[$(\mathbf{D_{2}})$] if the cubic polynomial (\ref{cubic}) has three distinct real roots and the corresponding singularities of $\mathcal{X}$ are a unique node between two saddle points,

\item[$(\mathbf{D_{3}})$] if the cubic polynomial (\ref{cubic}) has three distinct real roots and the corresponding singularities of $\mathcal{X}$ are three saddle points.
\end{itemize}
\end{definicion}

These conditions determine the qualitative behavior of the $\eta$-principal configuration around the $\eta$-umbilical point by the projection to the surface $S$ of the phase portraits in $L(\Phi,\eta)$ around these singularities. 
In our case, the conditions for having the Darbouxian principal configurations will be the same as for surfaces in $\mathbb{R}^3$ and $\mathbb{R}^4$. See figures 1 and 2. In figure 1 it is represented only one of the two families of $\eta$-principal curvature lines, following the presentation given in \cite{MR2105778}. 

\begin{figure}[h]
\label{fig} \centering
\includegraphics[scale=0.3]{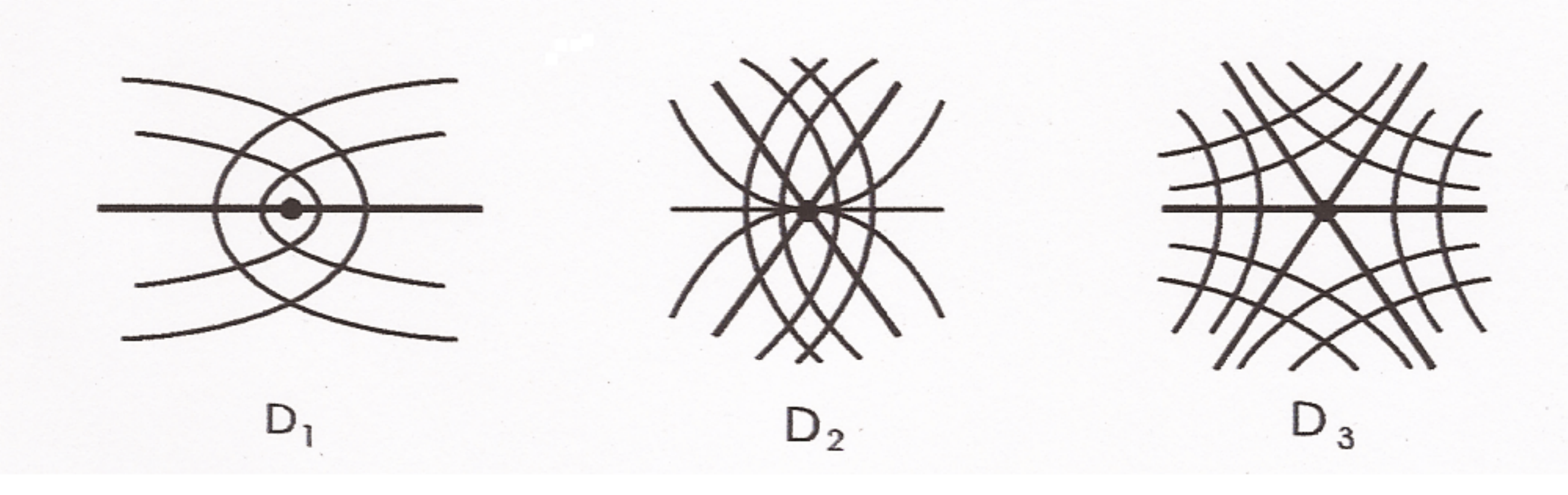}
\caption{Darbouxian principal configurations.}
\end{figure}

In the following section we will show that all three Darbouxian types are realizable for surfaces in null rotation hypersurfaces of the 4-Minkowski space $\mathbb{R}_1^4$.

\section{Spacelike surfaces in null rotation hypersurfaces of
$\mathbb{R}_1^4$.} \label{sec:nullrotation}

The goal of this section is to obtain a classification of the Darbouxian $\eta$-principal configurations for generic spacelike surfaces $S$ immersed in null rotation hypersurfaces of $\mathbb{R}_1^4$. 

The authors showed in \cite{MR3126943} that these null rotation hypersurfaces are classified in only three types, namely, null hyperplanes, 3-dimensional light cones and cylinders over 2-dimensional light cones, so we shall work with each type separately, beginning with the light cone $\Lambda^3$ to complete the study made in the last section of \cite{MR3126943}. 
Our main theorem \ref{teo:Darbouxianos} shows that the explicit dependence of the Darbouxian $\eta$-principal configurations on the parameters of the immersion (\ref{parametrizacionS}), choosing suitable functions $f(x,y)$ and vector fields $\xi$ in the parametrization (\ref{parametrizacionM}) for each type, follows the same rules in all types of null rotation hypersurfaces. 

\begin{teorema} \label{teo:Darbouxianos}
Let $M$ be a null rotation hypersurface of $\mathbb{R}_1^4$ and $S$ an spacelike surface immersed in $M$. Then, generically, an isolated simple $\eta$-umbilic point $p\in S$ which satisfy the transversality condition {\bf (T)} is of Darbouxian type ${\bf D_1, D_2}$ or ${\bf D_3}$ according to the following classification in terms of the parameters of differential equation (\ref{unojetABC}):
\begin{itemize}
\item[$(\mathbf{D_{1}})$] if $a/b>(c/2b)^2+2$ and $a>b>0$ 

or $a/b>(c/2b)^2+2$ and $a<b<0$,

\item[$(\mathbf{D_{2}})$] if $2<a/b<(c/2b)^2+2$ and $a<b<0$; $c>0$, 

or $1<a/b<2$ and $b>0$,

\item[$(\mathbf{D_{3}})$] if $2<a/b<(c/2b)^2+2$ and $a>b>0$; $c\neq 0$,

or $a/b<1$ and $c\neq 0$.
\end{itemize}
\end{teorema}

\noindent {\em Proof.} We give the proof for each type of null rotation hypersurface separately.

\subsection{Darbouxian umbilics in the light cone $\Lambda^3\subset \mathbb{R}^4_1$}\label{sec:3cono}

In \cite{MR3126943} the authors obtained the coefficients of (\ref{unojetABC}) for a generic spacelike surface $S$ immersed in the light cone $\Lambda^{3}$ of $\mathbb{R}_{1}^{4}$, in terms of parameters of the immersion, and they observed that Darbouxian $\eta$-principal configurations must appear generically. Now, we are going to write the coefficients of (\ref{unojetABC}) in a simpler form than the one given in \cite{MR3126943}. We observe first that, for the 3-light cone
\begin{equation*}
\Lambda^3=\{(t,x,y,z)\in \mathbb{R}_1^4: t^2=x^2+y^2+z^2\}-\{(0,0,0)\},
\end{equation*}
we have $(1,x,y,f(x,y))\in \Lambda^3$ if we choose $f(x,y)=\sqrt{1-x^2-y^2}$ in (\ref{parametrizacionS}) with $(x,y)\in U$, the unit disk in this case. 

The result is the map $\Phi:U\subset \mathbb{R}^2 \rightarrow \Lambda^3\subset \mathbb{R}_1^4$ defined by
\begin{equation}\label{eq:parametrizacionCONO}
\Phi (x,y)=(1+g(x,y))(1,x,y,\sqrt{1-x^2-y^2}),
\end{equation}
where $g$ is a smooth germ of a non constant function  $g:U\subset \mathbb{R}^{2}\rightarrow \mathbb{R}$ as in section \ref{sec:null}. Then we assume $g(0,0)=g_x(0,0)=g_y(0,0)=0$. Besides, we require that $g(x,y)\neq 1$ for all $(x,y)\in U$ because the origin is not a point of $\Lambda^3$. Then, for all $(x,y)\in U$,
\begin{center}
$\det
\left( 
\begin{array}{crl}
E &F \\
F &G \\
\end{array}
\right)
=\dfrac{(1+g(x,y))^4}{1-x^2-y^2}>0.
$
\end{center}
Therefore $\Phi$ is an immersion of $U$ into an spacelike surface $S=\Phi(U)\subset \Lambda^3\subset \mathbb{R}_1^4$. Moreover, using a rotation of the plane $(x,y)$ which eliminates the term $x^2y$, we may write the $3$-jet of $g(x,y)$ around $(0,0)$ as
\begin{equation*}
j^3g(x,y)=\frac{1}{2}(g_{xx}(0,0)x^{2}+2g_{xy}(0,0)xy+g_{yy}(0,0)y^{2})+\frac{a}{6}x^{3}+\frac{b}{2}xy^{2}+\frac{c}{6}y^{3}.
\end{equation*}
If $p=\Phi (0,0)$ is an isolated umbilical point of $S$ with respect to the null normal vector field $\eta$ defined by (\ref{eta}) in section \ref{sec:prelim}, the $\eta$-umbilicity at $p$ implies, by similar calculations to the ones given in the proof of lemma \ref{lema:1jetnormal}, that the second order derivatives of $g(x,y)$ at the origin satisfy
\begin{align*}
g_{xx}(0,0)&=g_{yy}(0,0)=k+1/2, \\
g_{xy}(0,0)&=0,
\end{align*}
where $k$ is the common value of the principal curvatures at $p$. Then we can write the components of the 3-jet of (\ref{eq:parametrizacionCONO}) as
\begin{equation*}
j^3\Phi _1(x,y)=1+\frac{1}{2}(k+1/2)(x^2+y^2)+\frac{a}{6}x^3+\frac{b}{2}xy^2+\frac{c}{6}y^3,
\end{equation*}

\begin{equation*}
j^3\Phi_2(x,y)=x+\frac{1}{2}(k+1/2)(x^3+xy^2),
\end{equation*}

\begin{equation*}
j^3\Phi_3(x,y)=y+\frac{1}{2}(k+1/2)(y^3+yx^2),
\end{equation*}

\begin{equation*}
j^3\Phi_4(x,y)=1+\frac{1}{2}(k-1/2)(x^2+y^2)+\frac{a}{6}x^3+\frac{b}{2}xy^2+\frac{c}{6}y^3,
\end{equation*}
and the 1-jet of the normal field $\eta$ coincides with (\ref{1jetnormal}). The 1-jet of the differential equation (\ref{unojetABC}) becomes
\begin{equation}\label{unojetCONO}
by\,dy^2+((a-b)x-cy)\,dx\,dy-by\,dx^2.
\end{equation}
Now, the isolated $\eta$-umbilical point $p=\Phi(0,0)\in S$ is simple and satisfies the transversality condition {\bf (T)} if $b(b-a)\neq 0$, which we will assume from now on; see remark \ref{condicionT}.

In \cite{MR985996}, theorem 0.1, it was proved that, if the 1-jet of a binary differential equation
\[
A(x,y)dy^2+2B(x,y)dydx-A(x,y)dx^2=0
\]
satisfy certain conditions, then there is a germ of a homeomorphism of the $xy$-plane onto itself taking its integral curves to the integral curves of one of the normal forms of binary differential equations which have local Darbouxian principal configurations $D_1, D_2, D_3$ in a neighborhood of isolated simple umbilical points (named Lemon, Monstar and Star in \cite{MR985996}). The conditions of theorem 0.1 in \cite{MR985996} are expressed in terms of the constants
\[
a_1=A_x(0,0), \;\;\;\; a_2=A_y(0,0),
\]
\[
b_1=B_x(0,0), \;\;\;\; b_2=B_y(0,0), 
\]
and say that the cubic polynomial
\begin{equation}\label{cubicBF}
\phi(z)=a_2 z^3+(2b_2+a_1)z^2+(-a_2+2b_1)z-a_1
\end{equation}
has no repeated roots, $a_1 b_2-a_2 b_1\neq 0$, and that $\phi(z)$ and $a_2 z^2+(a_1+b_2)z+b_1$ have no common roots. All these conditions are satisfied here due to the fact that our cubic polynomial (\ref{cubic}) coincides with $\phi(z)$ for our constants $a_1,a_2,b_1,b_2$, and has no repeated roots in the Darbouxian cases. Besides, our cubic polynomial (\ref{cubic}) and $a_2z^2+(a_1+b_2)z+b_1$ have no common roots if and only if the resultant of both polynomials vanishes (corollary 6, p. 180 of \cite{MR2975988}), which is satisfied if we have condition ({\bf T}) (see remark \ref{condicionT}) and the umbilical point is simple. Condition ({\bf T}) also implies that $a_1 b_2-a_2 b_1\neq 0$. We also observe that our conditions imply that $B^2-AC$ has a Morse singularity at the umbilical point and the hypothesis of proposition 3.2 in \cite{MR1328597} are also satisfied. Then the linear terms of our binary differential equation can be reduced to
\begin{equation}\label{normalform}
y\;dy^2+2(b_1 x+b_2 y)\;dx\;dy-y\;dx^2=0.
\end{equation}
See also theorem 4.1 of \cite{MR1328597} where the topological normal forms of (\ref{normalform}) are the Darbouxian ones and they are determined by the 1-jet of the functions $A(x,y), B(x,y)$. Moreover, the homeomorphism is a diffeomorphism away from the umbilical separatrices.

The corresponding cubic polynomial (\ref{cubic}) becomes 
\begin{equation*}
f_1(z)=b z^3-c z^2+(a-2b)z,
\end{equation*}
whose roots are
\begin{equation}
\begin{array}
[c]{l}
z_0=0,\\ 
\\
z_1=c/2b-\sqrt{(c/2b)^2-a/b+2},\\ 
\\
z_2=c/2b+\sqrt{(c/2b)^2-a/b+2},\\ \label{raices3cono}
\end{array}
\end{equation}
and the non zero eigenvalues (\ref{betas}) of the Lie-Cartan vector field (\ref{XLieCartan}) at each of these roots are
\begin{equation}\label{betas0}
\beta_2(z_0)=a-b, \ \ \ \ \ \ \ \ \ \ \ \ \ \ \ \ \beta_3(z_0)=2b-a,
\end{equation}
\begin{equation}\label{betas1}
\beta_2(z_1)=2b\Delta+a-b-c\sqrt{\Delta},\ \ \ \ \ \ \beta_3(z_1)=-2b\Delta+c\sqrt{\Delta},
\end{equation}
\begin{equation}\label{betas2}
\beta_2(z_2)=2b\Delta+a-b+c\sqrt{\Delta},\ \ \ \ \ \ \beta_3(z_2)=-2b\Delta-c\sqrt{\Delta},
\end{equation}
where
\begin{equation*}
\Delta=(c/2b)^2 - a/b + 2.
\end{equation*} 
. By definition \ref{Darbouxian} this point will be Darbouxian if $b\neq 0$, $a/b\neq 1$ and the discriminant $\Delta$ is different from $0,(c/2b)^2$, that is to say that $a/b\neq 2$ and also $a/b\neq (c/2b)^2+2$. Therefore, by definition \ref{D1D2D3}, the Darbouxian types $D_1,D_2,D_3$ can only exist in four different cases: 

\bigskip

{\bf Case I.} $a/b>(c/2b)^2+2$.

\bigskip

We have $\Delta<0$ and $z_0=0$ is the only real root. The corresponding singularity of $\mathcal{X}$ will be a saddle point if and only if $\beta_2(0)\beta_3(0)<0$. By (\ref{betas0}) this means $(a-b)(2b-a)<0$. Because $a/b>2$ we have Darbouxian type $D_1$ in this case if we also have $a>b>0$ or $a<b<0$.

\bigskip

{\bf Case II.} $2<a/b<(c/2b)^2+2$.

\bigskip

In this case $\Delta>0$ and we have three distinct real roots given by (\ref{raices3cono}).  If $c/b>0$ these roots satisfy $z_0<z_1<z_2$ and if $c/b<0$ then $z_1<z_2<z_0$. Consequently, Darbouxian types can appear in two subcases. For $c/b>0$ the singularities of $\mathcal{X}$ corresponding to the roots $z_0$ and $z_2$ must be of saddle type, while if $c/b<0$ this must occur for the roots $z_1$ and $z_0$.   

\bigskip

By hypothesis $a/b>2$. Then (\ref{betas0}) implies that the root $z_0$ corresponds to a saddle point if $a>b>0$. Then $c/b>0$ if $c>0$. Because $z_0<z_1<z_2$ we must have a saddle point for $z_2$. Using (\ref{betas2}) we have
\begin{equation*}
\beta_2(z_2) \beta_3(z_2)=(b-a)(2b\Delta+c\sqrt{\Delta})-(2b\Delta+c\sqrt{\Delta})^2.
\end{equation*}
Now observe that $2b\Delta \pm c\sqrt{\Delta}$ never vanishes because $a/b\neq 2$. Then $z_2$ corresponds to a saddle point if $b-a<2b\Delta+c\sqrt{\Delta}$. This is satisfied if $a>b>0$ and $c>0$. For $z_1$ we can see from (\ref{betas1}) that it will be of saddle type if $b-a<2b\Delta-c\sqrt{\Delta}$ which is true for $a>b>0$ and $c>0$. Consequently, we have Darbouxian principal configuration $D_3$ in this case if $a>b>0$ and $c>0$. On the other hand, if $a>b>0$ and $c<0$ we also have $D_3$ because $c/b<0$ implies that $z_1<z_2<z_0$ and all three roots correspond to saddle singularities of $\mathcal{X}$ by similar arguments.

\bigskip

If $a<b<0$ and $c>0$ then $c/b<0$ and $z_1<z_2<z_0$. The hypothesis $a/b>2$ and (\ref{betas0}) imply that $z_0$ corresponds to a saddle point. The root $z_1$ also represents a saddle point of 
$\mathcal{X}$ with these inequalities, due to (\ref{betas1}). And using (\ref{betas2}) we can see that $z_2$ corresponds to a node of $\mathcal{X}$. Therefore $a<b<0$ and $c>0$ gives us Darbouxian type $D_2$. 

\bigskip

If $a<b<0$ and $c<0$ then $c/b>0$ and $z_0<z_1<z_2$. The root $z_2$ can not correspond to a saddle singularity of $\mathcal{X}$ with these conditions unless $\Delta$ is negative, which contradicts the hypothesis of Case II. 

\bigskip

{\bf Case III.} $1<a/b<2$.

\bigskip

In this case (\ref{betas0}) implies that the root $z_0$ represents a node singularity of $\mathcal{X}$. Besides, $\Delta>(c/2b)^2$ and $z_1<z_0<z_2$. Using (\ref{betas1}) and (\ref{betas2}) is easy to see that the corresponding singularity for $z_1$ is of saddle type and for $z_2$ is of saddle type whenever $b>0$. Then we have Darbouxian type $D_2$ if $b>0$.

\bigskip

{\bf Case IV.} $a/b<1$.

\bigskip

If $a/b<1$ then $\Delta>(c/2b)^2$. We have two subcases: $c/b>0$ or $c/b<0$. In any case the roots (\ref{raices3cono}) satisfy $z_1<z_0<z_2$. Using (\ref{betas0}) we have that $z_0$ always represents a saddle point of $\mathcal{X}$ because $a/b<2$. Now, using (\ref{betas1}) it  is easy to see that for $c/b>0$ the root $z_1$ corresponds to a saddle of $\mathcal{X}$ if $b>0$. By (\ref{betas2}) the same is true for the root $z_2$. Therefore $b>0$ and $c>0$ implies that the principal configuration is of type $D_3$. If $c/b>0$ and $b<0$ then $c<0$. Using (\ref{betas1}) and (\ref{betas2}) we have that both roots $z_1$ and $z_2$ represent a saddle of $\mathcal{X}$. Consequently, $b<0$ and $c<0$ give us Darbouxian type $D_3$. By similar arguments, we have $D_3$ type if $bc<0$.

\subsection{Darbouxian umbilics in a null hyperplane of $\mathbb{R}^4_1$}\label{sec:planoNULO}

If we choose $f(x,y)\equiv 0$ in (\ref{parametrizacionS}) we obtain a generic spacelike surface $S$ immersed in a null hyperplane $M^3$ of $\mathbb{R}_1^3$ parametrized by 
$\varphi: \mathbb{R}^3\rightarrow M^3\subset \mathbb{R}_1^4$,
\begin{equation}\label{parametrizacionPLANOnulo}
\varphi(x,y,t)=(1,x,y,0)+t\xi,
\end{equation}
where $\xi$ is the null vector $(1,0,0,1)$. The tangent vectors $\varphi_x$, $\varphi_y$ are spacelike and 
$\varphi_t=\xi$. In fact, $M^3$ is a translation by $e_0=(1,0,0,0)$ of the hyperplane of $\mathbb{R}_1^4$ generated by the three linearly independent vectors $e_1=(0,1,0,0)$, $e_2=(0,0,1,0)$ and $\xi$. Now, letting $t=g(x,y)$ in (\ref{parametrizacionPLANOnulo}) as in subsection \ref{sec:3cono}, we obtain a surface $S\subset M^3$ parametrized by
\begin{equation}\label{SenPLANOnulo}
\Phi(x,y)=(1+g(x,y),x,y,g(x,y))
\end{equation}
with $(E G -F^2)(x,y)\equiv 1$. Therefore, $S$ is spacelike. As in lemma \ref{lema:1jetnormal}, the null vector $\eta_0=(-1/2,0,0,1/2)$ is orthogonal to the surface $S$ at $p=\Phi(0,0)$ with $\langle \eta_0, \xi \rangle = 1$ and the umbilicity condition implies that 
\begin{equation*}
e_\eta(0,0)=g_{xx}(0,0)=k, \ \ \ \ f_\eta(0,0)=g_{xy}(0,0)=0, \ \ \ \ g_\eta(0,0)=g_{yy}(0,0)=k,
\end{equation*}
where $k$ is the principal curvature at the isolated $\eta$-umbilical point $p$. Using a rotation on the plane $(x,y)$ which eliminates the term $x^2y$, we may write the $3$-jet of $g(x,y)$ around the origin as
\begin{equation*}
j^3g(x,y)=\dfrac{k}{2}(x^2+y^2)+\dfrac{a}{6}x^3+\dfrac{b}{2}xy^2+\dfrac{c}{6}y^3.
\end{equation*}
Therefore, from (\ref{SenPLANOnulo}),
\begin{equation}\label{3jetSenPLANOnulo}
j^3\Phi(x,y)=(1+j^3g(x,y),x,y,j^3g(x,y)).
\end{equation}
The 1-jet of $\eta$ turns out to be the same given in (\ref{1jetnormal}). With the 3-jet of our parametrization (\ref{3jetSenPLANOnulo}) and the 1-jet of $\eta$, the differential equation (\ref{unojetABC}) becomes exactly the same as the one given in (\ref{unojetCONO}), the cubic polynomial (\ref{cubic}) also coincides, having the same roots (\ref{raices3cono}). Finally, the same eigenvalues (\ref{betas0}), (\ref{betas1}), (\ref{betas2}) are obtained here for the corresponding  Lie-Cartan vector field 
$\mathcal{X}$. Consequently, we have the classification of Darbouxian $\eta$-umbilical points in terms of the parameters $a,b,c$ stated in theorem \ref{teo:Darbouxianos} for a generic spacelike surface $S$ immersed in a generic null hyperplane of $\mathbb{R}_1^4$.

\subsection{Darbouxian umbilics in a cylinder $\Lambda^2 \times \mathbb{R}\subset 
\mathbb{R}_1^4$}\label{sec:cilindro}

Our last type of null rotation hypersurface can be given generically by choosing $f(x,y)=\sqrt{1-x^2}$ in (\ref{parametrizacionM}) which gives us $\xi(x,y)=(1,x,0,\sqrt{1-x^2})$ and the parametrization (\ref{parametrizacionS}) turns out to be $\Phi:(-1,1)\times \mathbb{R}\rightarrow \mathbb{R}_1^4$, defined by
\begin{equation}\label{parametrizacionCILINDRO}
\Phi(x,y)=(1,x,y,\sqrt{1-x^2})+g(x,y)\xi(x,y),
\end{equation}
for a function $g$ as in subsection \ref{sec:3cono}. Observe that the components of $\Phi$ which do not contain the variable $y$ satisfy the equation of the light cone:
\begin{equation*}
\Lambda^2=\{ (t,x,0,z)\in \mathbb{R}_1^4: t^2=x^2+z^2 \}-\{(0,0,0)\}.
\end{equation*}
Then $\Phi$ parametrizes a surface $S$ immersed in a generic cylinder $\Lambda^2 \times \mathbb{R}$. This surface is spacelike because
\begin{equation*}
(E G - F^2)(x,y)=\dfrac{(1+g(x,y))^2}{1-x^2},
\end{equation*}
which is positive for all $(x,y)\in (-1,1) \times \mathbb{R}$. Now, as in previous calculations, imposing the conditions
\begin{equation}\label{condicionesETA}
\langle \eta_0, \xi \rangle =1, \ \ \langle \eta_0, \Phi_x \rangle = \langle \eta_0, \Phi_y \rangle = \langle \eta_0, \eta_0 \rangle =0,
\end{equation}
at the origin $(0,0)$ gives us the null normal vector $\eta_0=(-1/2,0,0,1/2)$. This vector can be extended to the same normal field (\ref{1jetnormal}) by a similar process. The first step is to choose an adequate rotation of the $(x,y)$-plane in order to eliminate some mixed term of the 3-jet of $g(x,y)$. In this case we choose the rotation such that
\begin{equation*}
j^3g(x,y)=\dfrac{1}{2}(g_{xx}(0,0) x^2+2 g_{xy}(0,0) x y + g_{yy}(0,0)) + \dfrac{a}{6}x^3+\dfrac{b}{2}x y^2 + \dfrac{c}{6}y^3.
\end{equation*}
The next step is to impose the umbilicity conditions for $p=\Phi(0,0)$ which gives us the coefficients of the $\eta$-second fundamental form at $(0,0)$ as follows:
\begin{equation*}
e_\eta(0,0)=k, \ \  f_\eta(0,0)=0, \ \ g_\eta(0,0)=k,
\end{equation*}
where $k$ is the $\eta$-principal curvature at the umbilic point. Consequently, 
\begin{equation*}
g_{xx}(0,0)=k+1/2, \ \ g_{xy}(0,0)=0, \ \ g_{yy}(0,0)=k.
\end{equation*}
Then, our parametrization (\ref{parametrizacionCILINDRO}) has 3-jet with the following components:
\begin{eqnarray*}
j^3\Phi_1(x,y)&=&1+\dfrac{1}{2}((k+1/2)x^2+ky^2)+\dfrac{a}{6}x^3+\dfrac{b}{2}xy^2+
\dfrac{c}{6}y^3, \\
j^3\Phi_2(x,y)&=&x+\dfrac{1}{2}(k+1/2)x^3+\dfrac{k}{2}xy^2, \\
j^3\Phi_3(x,y)&=&y, \\
j^3\Phi_4(x,y)&=&1+\dfrac{1}{2}((k-1/2)x^2+ky^2)+\dfrac{a}{6}x^3+\dfrac{b}{2}xy^2+\dfrac{c}{6}y^3,
\end{eqnarray*}
which can be introduced in (\ref{condicionesETA}), replacing $\Phi$ by its 3-jet $j^3\Phi$ and $\eta_0$ by $j^1\eta$ as in lemma \ref{lema:1jetnormal}, to obtain the 1-jet (\ref{1jetnormal}) of $\eta(x,y)$. Then, using this 3-jet of $\Phi$ and the 1-jet of $\eta$ we obtain the functions $A,B,C$ of the differential equation (\ref{ABC}), whose 1-jets turn out to be
\begin{eqnarray*}
j^1A(x,y)&=&by, \\
j^1B(x,y)&=&(a-b)x-cy, \\
j^1C(x,y)&=&-by. 
\end{eqnarray*}
Consequently, as in subsection \ref{sec:planoNULO}, if $b(b-a)\neq 0$ the Darbouxian classification in terms of the parameters $a,b,c$, coincides in this case with the first one. \Qed

\section*{Acknowledgements}

The first author is grateful to CONACYT for the grant 457490 and Facultad de Ciencias UNAM for the warm hospitality during the sabbatical year in which this work was developed, under Project FMAT-2016-0013 of UADY. 

The second author was partially supported by UNAM, under Project PAPIIT-DGAPA IN113516 and also by UADY, under the {\em C\'atedra Dr. Eduardo Urz\'aiz Rodr\'iguez}. 

The third author is grateful to CIMAT-M\'erida for the warm hospitality during the sabbatical year in which this work was developed. He was partially supported by FMAT-UADY, under Project PROFOCIE 2015-12-1918.

\bibliographystyle{plain}

\end{document}